\documentclass[11pt,reqno]{amsart}
\usepackage{amsbsy,amsfonts,amsmath,amssymb,amscd,amsthm}
\usepackage{mathrsfs}
\usepackage{subfigure,enumitem,color,graphicx}
\usepackage[a4paper,text={6.1in,8in},centering,includefoot,foot=1in]{geometry}
\usepackage{pxfonts}
\usepackage[colorlinks=true,linkcolor=blue,citecolor=green]{hyperref}
\usepackage{srcltx}

\small\normalsize
\theoremstyle{plain}
\newtheorem{maintheorem}{Theorem}
\newtheorem{maincorollary}[maintheorem]{Corollary}
\newtheorem{theorem}{Theorem}[section]
\newtheorem{lemma}[theorem]{Lemma}
\newtheorem{proposition}[theorem]{Proposition}
\newtheorem{corollary}[theorem]{Corollary}

\newtheorem{definition}[theorem]{Definition}
\newtheorem{remark}[theorem]{Remark}
\newtheorem{example}[theorem]{Example}

\theoremstyle{definition}

\DeclareMathOperator{\IFS}{IFS} \DeclareMathOperator{\diam}{diam}
\newcommand{\eqdef}{\stackrel{\scriptscriptstyle\rm def}{=}}

\let\oldtocsection=\tocsection
\let\oldtocsubsection=\tocsubsection
\renewcommand{\tocsection}[2]{\hspace{0em}\bf\oldtocsection{#1}{#2}}
\renewcommand{\tocsubsection}[2]{\hspace{1em}\oldtocsubsection{#1}{#2}}
\let\oldtocsubsubsection=\tocsubsubsection
\renewcommand{\tocsubsubsection}[2]{\hspace{2em}\oldtocsubsubsection{#1}{#2}}

\begin{document}

\title{Expanding actions: Minimality and Ergodicity}


\author[Barrientos]{Pablo G. Barrientos}
\address{\footnotesize \centerline{Instituto de Matem\'atica e Estat\'istica, UFF}
   \centerline{Rua M\'ario Santos Braga s/n - Campus Valonguinhos, Niter\'oi,  Brazil}}
\email{barrientos@id.uff.br}

\author[Fakari]{Abbas Fakhari}
\address{\footnotesize \centerline{Department of Mathematics, Shahid
Beheshti University} \centerline{G.C.Tehran 19839, Tehran, Iran}}
\email{a\_fakhari@sbu.ac.ir}

\author[Malicet]{\\ Dominique Malicet}
\address{\footnotesize \centerline{Departamento de Matem\'atica PUC-Rio}
    \centerline{Marqu\^es de S\~ao Vicente 225, G\'avea, Rio de Janeiro 225453-900, Brazil}}
\email{malicet@mat.puc-rio.br}

\author[Sarizadeh]{Aliasghar Sarizadeh}
\address{\footnotesize \centerline{Department of Mathematics, Ilam University}
\centerline{Ilam, Iran}} \email{a.sarizadeh@mail.ilam.ac.ir}


\maketitle
\begin{abstract}
We prove that every expanding minimal semigroup action of $C^1$
diffeomorphisms of a compact manifold (resp.~$C^{1+\alpha}$
conformal) is robustly minimal (resp.~ergodic with respect to
Lebesgue measure). We also show how, locally, a blending region
yields the robustness of the minimality and implies ergodicity.
%We also obtain results about the equivalence to the Lebesgue measure
%and fiberwise density of a
%stationary measure.
\end{abstract}

\vspace{1cm}

%\tableofcontents

\section{Introduction}
Minimality and ergodicity are two popular themes in dynamical
systems. Minimality can be thought of as a property involving some
complexity for the orbits of an action coming from the fact that
the action is irreducible from topological point of view.
%A dynamical system is said to be \emph{minimal} if every closed invariant subset
%is either empty or coincides with the whole space.
From  probabilistic point of view, the
counterpart to this notion corresponds to ergodicity.  %Equivalently, if
%every point has dense orbit.  is ergodicity.
%An invariant probability measure $\mu$ is called \emph{ergodic} if
%every invariant measurable set is either of zero or full
%$\mu$-measure.
%Actually, ergodicity concerns the behavior of almost every point,
%and not of all the points.
It is natural to ask to what extent the properties of ergodicity
and minimality are related.
%A difeomorphism
%of a compact  manifold always leaves the Riemannian volume (also
%called the Lebesgue measure) quasi-invariant, and one can ask if a
%given diffeomorphism is ergodic with respect to the Lebesgue measure
%(below ergodic or Lebesgue-ergodic for short) or not.
In general, ergodicity does not imply minimality. Indeed, one can
easily construct examples of ergodic group actions having a global
fixed points. A natural question arises in the opposite direction:
\begin{quotation}
\begin{center}
 {\it Under which conditions a smooth minimal action on
a compact manifold is Lebesgue-ergodic?}
\end{center}
\end{quotation}
The example of Furstenberg~\cite{F61} shows that minimality does
not imply ergodicity in general,. Namely, Furstenberg constructs a
minimal analytic diffeomorphism of $\mathbb{T}^2$ preserving the
Lebesgue measure but not ergodic. In the case of unit circle, a
classical theorem by Katok and Herman states that a
$C^1$-diffeomorphism with derivative of bounded variation is
ergodic provided its rotation number is irrational (see
\cite{Ka95,H75}). Contrarily in~\cite{OR01}, the authors
constructed minimal $C^1$-diffeomorphisms on the unit circle which
are not ergodic.

The concept of expanding action was initially conceived to study
the ergodicity of the circle actions. In light of the rich
consequences for expanding action on the unit circle (see,
\cite{Na04,Hu}), it is natural to consider an extended notion. We
start with a precise definition of this extension. %in general case.
\begin{definition}
An action of a semigroup $\Gamma$ of diffeomorphims of compact
manifold $M$ is \emph{expanding} if for every $z\in M$ there is
$g\in \Gamma$ such that $m(Dg^{-1}(z))>1$, where
$m(T)=\|T^{-1}\|^{-1}$ denotes the co-norm of a linear
transformation $T$. %, i.e., $m(T)=\|T^{-1}\|^{-1}$.
%$$
%\text{for every $z\in M$ there is $g\in \Gamma$ such that
%$m(Dg^{-1}(z))>1$,}
%$$
%where $m(T)=\|T^{-1}\|^{-1}$ denote the co-norm of a linear
%transformation $T$. %, i.e., $m(T)=\|T^{-1}\|^{-1}$.
\end{definition}

Notice that by a very simple compactness type argument one can
show that the expanding property is \emph{robust}, i.e., persists
under the perturbation of the generators. This work is devoted to
prove the robust minimality of minimal expanding action and the
ergodicity of  conformal action with respect to the
Lebesgue measure.

Minimality and ergodicity should state that the orbit of every
non-trivial set, from the topological and the measure theoretical
point of view, fills most of the space. In order to introduce
ergodicity and minimality for group and semigroup actions we
should first deal with the meaning of \emph{invariant set}. %This
%depends strongly on the dynamical systemize are working with.
Observe that the orbits are both forward and backward invariant
under the group actions while they are only forward invariant in
the case of semigroup actions. According to this observation, we
say that a set $A\subset M$ is \emph{(totally) invariant}  for a
group action, generated by a family $\mathcal{F}$, if $f(A)=A$,
for all $f\in \mathcal{F}$. However, we need to slightly change
the definition in the case of semigroup action. We say that $A$ is
\emph{(forward) invariant} for the semigroup action, generated by
$\mathcal{F}$, if $f(A)\subset A$, for all $f\in \mathcal{F}$.
Nevertheless, we can unify both saying that $A$ is
\emph{$\Gamma$-invariant} if $f(A)\subset A$ for all $f\in\Gamma$
where $\Gamma$ denotes the group/semigroup generated by
$\mathcal{F}$. Since every group can be seen as a semigroup
choosing an appropriate set of generators, in what follows we will
assume that $\Gamma$ is a semigroup.

%In this paper we will discuss the expanding property ranging from
%the definition above to two main theorems below and their
%consequences. To formalize our results, we will first define the
%notion of minimality.

The action of $\Gamma$ on $M$ is {\it minimal} if every closed
$\Gamma$-invariant set is either empty or coincides with the whole
of manifold. If $\Gamma$ is finitely generated, {\it robust
minimality} means that the minimality of the action does not
disappear perturbing the initial generators. There are several
examples of robustly minimal actions: in dimension one by
\cite{GI00} and by \cite{GHS10,HN13} for every boundaryless
compact manifold. In our first theorem, we identify the expanding
property as an important feature in determining the robust
minilality.

%Expansive actions of $\mathbb{Z}$ on compact Riemannian manifold $M$
%are well-understood area of dynamics \cite{38,30}. For example, it is
%simple to show that the topological entropy of an expansive
%diffeomorphism must be positive. We are now ready to state the first result of this paper.
\begin{maintheorem}
\label{thmA} Every expanding minimal finitely generated semigroup
action of $C^{1}$ diffeomorphisms on a compact manifold is
$C^1$-robustly minimal.
\end{maintheorem}

In fact, we will prove that any expanding semigroup of
diffeomorphisms (not necessarily finitely generated) acting
minimally has a finitely generated sub-semigroup whose action is
also expanding and robustly minimal (see
Theorem~\ref{thmA-generalize}).

The second main result is due to the ergodicity of minimal
actions. The definition of ergodicity can be naturally extended to
a \emph{quasi-invariant} measure which is measure whose
push-forward is absolutely continuous with respect to itself.
Here, we focus on the Lebesgue measure which is quasi-invariant for
$C^1$ (local) diffeomorphisms. The action of $\Gamma$ is
\emph{ergodic} with respect to a quasi-invariant probability
measure $\mu$  if $\mu(A)\in \{0,1\}$, for all $\Gamma$-invariant
set $A\subset M$. Since the $C^1$-regularity is not sufficient to
conclude the ergodicity from the minimality (\cite{OR01}), the majority of the results,
obtained up to now, are in the context of $C^{1+\alpha}$-regularity. %diffeomorphisms ($C^{1}$
%diffeomorphisms with $\alpha$-H\"older derivatives with
%$\alpha>0$).
For instance, following essential idea
of~\cite{SS85}, Navas proved in \cite{Na04} that every expanding
minimal action of a group of $C^{1+\alpha}$ diffeomorphisms of the
unit circle is Lebesgue-ergodic. This is the main motivation behind the
next theorem. We extend the result for semigroup actions in any
compact Riemannian manifold assuming conformality.

Recall that a diffeomorphism $g$ of compact Riemannian manifold $M$
is a \emph{conformal} map if there exists a function $a:M\to \mathbb{R}$ such that for all $x
\in M$ we have that $Dg(x)= a(x) \, \mathrm{Isom}(x)$, where
$\mathrm{Isom}(x)$ denotes an isometry of $T_xM$. Clearly,
$a(x)=\|Dg(x)\|=m(Dg(x))$, for all $x\in M$.

\begin{maintheorem}
\label{thmB} Every expanding minimal semigroup action of
$C^{1+\alpha}$ conformal diffeomorphisms of a compact Riemannian manifold is
ergodic with respect to the Lebesgue measure.
\end{maintheorem}

%As it well know, Luioville's theorem and it generalizations
%severely limits the variety of possible (smooth) conformal
%mappings on any conformal manifold of dimension $d\geq 3$.
Notice that in one dimensional case any diffeomorphism is conformal.
Hence combining Theorems~\ref{thmA} and~\ref{thmB} we get the following.

\begin{maincorollary}
Every expanding minimal finitely generated semigroup action of
$C^{1+\alpha}$ diffeomorphisms of the circle is
$C^{1+\alpha}$-robustly ergodic with respect to the
Lebesgue measure.
\end{maincorollary}

By \emph{$C^{1+\alpha}$-robust} ergodicity we mean that this
ergodicity property persists under $C^{1}$-pertur\-ba\-tion among
the generators with H\"older continuous derivative.  A completely
non-trivial example of a persistent ergodic action in higher
dimension can be found in~\cite{DK07}. There, the authors show
that if the rotation $R_1,\ldots,R_m$ generate
$\mathbb{SO}_{d+1}$, for even number $d$, then their action on the
sphere $\mathbb{S}^d$ is \emph{stable} ergodic with respect to the
Lebesgue measure. That is, if $f_i$ are $C^\infty$ diffeomorphisms
sufficiently close to $R_i$ and \emph{preserve} the Lebesgue
measure then the action generated by $f_i$ is also ergodic.

In this paper we identify a local feature, called  \emph{blending region}, allowing us to get robustness of
minimality and ergodicity (see Theorem~\ref{pro2}). As an
application, we get, for the first time as far as
we know, examples of $C^{1+\alpha}$-robustly Lebesgue-ergodic
semigroup actions in any surface.

\begin{maincorollary}
\label{rem:ergodic-surface} Any boundaryless compact surface
admits a $C^{1+\alpha}$-robustly minimal Lebesgue-ergodic
%with respect to Lebesgue
finitely generated semigroup action whose
generators are not necessarily conformal maps.
\end{maincorollary}

\vspace{0.5cm} \noindent {\bf This work is organized as follows:}
In Section~\ref{Sec-2} we compare the, a priori, different notions
of mini\-ma\-li\-ty (resp.~ergodicity) in the case of one
generator. In Section~\ref{Sec-minimality}, we prove
Theorem~\ref{thmA}. The proof of Theorem~\ref{thmB} is handled
in Section~\ref{ss:proofB}. %Previously in
%Subsection~\ref{Sec-uniform-expanding}, we study the ergodicity
%of random uniformly expanding semigroups actions of local
%diffeomorphisms.
The rest of the paper is dedicated to the study of local
features allowing us to construct new classes of ergodic expanding minimal actions.

\section{General observations in the case of one generator}
\label{Sec-2}
%a single diffeomorphisms.
%A weaker notion than the minimality is the transitivity. The  action of $\Gamma$ is \emph{topologically transitive} if for every pair of non-empty open set $U, V$ in $M$ there exits $g\in\Gamma$ such that $U\cap g^{-1}(V)\not=\emptyset$. Equivalently if there exists a point with dense orbit.
%Minimal actions are transitive. Even, observe that every  transitive semigroup action of homeomorphisms of $M$ is forward transitive (i.e., the semigroup generated by the family of inverse maps is also transitive).
%
%In order to introduce ergodicity we recall that the notion of quasi-invariant measure. A probability measure $\mu$ is called \emph{quasi-invariant} for $\Gamma$ if $f_\ast \mu$ is absolutely continuous with respect to $\mu$ for all $f\in\Gamma$.
%
%\begin{definition}
%An action of a group (resp. of a semigroup) generated by a family $\mathcal{F}$ of homeomorphisms of $M$ is \emph{ergodic} with respect to a quasi-invariant measure $\mu$ if $\mu(A)\in \{0,1\}$ for
%all measurable $\mathcal{F}$-invariant (resp. forward $\mathcal{F}$-invariant) subset $A\subset M$.
%\end{definition}
%\subsection{}
Given a diffeomorphism $f$ of a compact manifold $M$, one can
study the iterations of the map from two points of view. The
first, consists to consider full orbits, i.e., forward and
backward iterations of $f$, called \emph{$\mathbb{Z}$-action} of
$f$. The second only considers forward iterations of $f$, called
\emph{$\mathbb{N}$-action} of $f$ (or \emph{cascade}). Since we
have introduced different notion of invariance for group and
semigroup actions, a priori, one could expect different type of minimality and
for the $\mathbb{Z}$-action
and the $\mathbb{N}$-action. However, it is
easy to prove the following: %Ergodic Theory and topological dynamics (Gerard Meurant)
\begin{proposition}
The $\mathbb{Z}$-action of $f$ is minimal if and only if the
$\mathbb{N}$-action of $f$ is minimal.
\end{proposition}
%\begin{proof}
%It is clear that the minimality of the cascade implies the
%minimality of the $\mathbb{Z}$-action. Reciprocally, by
%contradiction, suppose that the $\mathbb{Z}$-action of $f$ is
%minimal but not the cascade. Let $x$ be a point whose forward
%orbit is not dense. Since $M$ is a compact manifold, the
%$\omega$-limit of $x$ is a closed non-empty set invariant for the
%$\mathbb{Z}$-action. Since the forward orbit of $x$ is not dense
%then the $\omega$-limit of $x$ is different that the whole space.
%Then the $\mathbb{Z}$-action cannot be minimal contradicting our
%assumption. This conclude the proposition.
%\end{proof}
As in the case of minimality, the ergodicity of the
$\mathbb{N}$-action implies the ergodicity of the
$\mathbb{Z}$-action. In the case of invariant measures, one can
also get the converse. Recall that a measure $\mu$ on $M$ is said
to be \emph{invariant} under $f$ if $\mu(f^{-1}(A))=\mu(A)$ for
all measurable set $A\subset M$, that is $f_*\mu=\mu$, in terms of
the push-forward. The following proposition states that if $\mu$  is an invariant measure of $f$ then
$\mu(A)\in \{0,1\}$ for all measurable set so that $f(A)=A$
if and only if $\mu(A)\in \{0,1\}$ for all measurable set so that $f(A)\subset A$.

%The following proposition is a natural extension of a well known result in abstract ergodic theory (\cite{kaku,KIF86}).
\begin{proposition}
%\label{prop:erg}
Let $\mu$ be an invariant measure of $f$.
%for a
%diffeomorphism $f$ of a compact manifold $M$. Then
The $\mathbb{Z}$-action of $f$ is ergodic with respect  $\mu$ if
and only if the $\mathbb{N}$-action of $f$ is ergodic with respect
to $\mu$.
\end{proposition}
\begin{proof}
It suffices to see that the ergodicity of the $\mathbb{Z}$-action
implies the ergodicity of the cascade. In order to do this, consider
a measurable set $A\subset M$ such that $f(A)\subset A$. We will
prove that $\mu(A)\in \{0,1\}$. Consider
$$
\Theta=\bigcup_{n=0}^\infty f^{-n}(A).
$$
Observe that $f(\Theta)=\Theta$. Indeed, by assumption  $A\subset
f^{-1}(A)$ and therefore, $f^{-1}(\Theta)=f^{-1}(A)\cup
f^{-2}(A)\cup \dots= A \cup f^{-1}(A) \cup f^{-2}(A) \cup
\dots=\Theta$. By the ergodicity of the $\mathbb{Z}$-action with
respect to $\mu$, $\mu(\Theta)\in \{0,1\}$. If $\mu(\Theta)=0$ then
$\mu(A)=0$. On the other hand, if $\mu(\Theta)=1$ then the measure
of
$$
M\setminus\Theta = \bigcap_{n=0}^\infty f^{-n}(A^c)
$$
is zero, where $A^c=M\setminus A$. Now, by the inclusion
$f^{-(n+1)}(A^c) \subset f^{-n}(A^c)$, for every $\varepsilon>0$
there exists $n_0\in \mathbb{N}$ such that
$\mu(f^{-n}(A^c))<\varepsilon$, for any $n\geq n_0$. By the
invariance of $\mu$, we get that $\mu(A^c)<\varepsilon$, for all
$\varepsilon>0$ and consequently $\mu(A)=1$. This concludes the
proof of the proposition.  \end{proof}

As far as the authors know, if $\mu$ is a quasi-invariant but not invariant
then the problem of equivalence between the ergodicity for
$\mathbb{Z}$-actions and for $\mathbb{N}$-action remains open. For the actions generated by more than one map, the
minimality/ergodicity  of the group action does not imply, in
general, the minimality/ergodicity of the semigroup action. A
simple example can be constructed as follows:
\begin{example}
Consider a pair of diffeomorphisms $f_0, f_1$ of the circle
$C^2$-close enough to rotations with zero rotation number, no
fixed points in common and with an $ss$-interval (compact interval
whose endpoints are consecutive attracting fixed points, one of
$f_0$ and one of $f_1$). According to Theorem~A and
Theorem~5.4 in~\cite{BR13}, the group action generated by these two maps
is minimal while the semigroup action cannot be minimal. In fact,
the $ss$-interval is a non-empty closed invariant set for the
semigroup action different of the circle $S^1$. Since this set has
positive but not full Lebesgue measure, the semigroup action
cannot be also Lebesgue-ergodic. However, according
to Theorem~D in~\cite{Na04}, the group action is ergodic with
respect to the Lebesgue measure (see also~\cite{BR13}).
\end{example}
%Finally, we would mention the meaning of backward ergodicty of a
%semigroup action. Let $\mu$ be a quasi-invariant measure for the
%action of the group generated by a family of diffeomorphisms
%$\mathcal{F}$. Observe that $\mu$ is also a quasi-invariant measure
%for both, the semigroup generated by $\mathcal{F}$ and by
%$\mathcal{F}^{-1}$. Here $\mathcal{F}^{-1}$ denotes the family of
%inverse maps of $\mathcal{F}$. The semigroup action generated by
%$\mathcal{F}$ is \emph{backward ergodic} respect to $\mu$ if the
%semigroup action generated by $\mathcal{F}^{-1}$ is ergodic.
%\begin{remark}
%If the semigroup action generated by $\mathcal{F}$ is ergodic with respect to a quase-invariant measure $\mu$ for the group action generated by $\mathcal{F}$ then it so is backward ergodic with respect to~$\mu$.
%\end{remark}
\section{Minimality of expanding actions}
\label{Sec-minimality}
%\subsection{Proof of Theorem~\ref{thmA}}
We get the proof of Theorem~\ref{thmA} as a corollary of the following
more general result.
\begin{theorem}
\label{thmA-generalize} Let  $\Gamma$ be a semigroup of $C^1$
diffeomorphisms acting on a compact manifold $M$. Assume that the action is expanding and minimal.
Then, there exists a finite set $\mathcal{G}\subset\Gamma$
such that the semigroup action generated by $\mathcal{G}$ is
$C^1$-robustly minimal.
\end{theorem}
In order to prove theorem above, we need the following lemma
which is obtained straightforwardly by a compactness argument.
\begin{lemma} \label{lem1-cpt}
The action of $\Gamma$ on $M$ is expanding if and only if there
are $h_1,\dots,h_k \in \Gamma$, open balls $B_1,\dots,B_k$ in $M$
and a constant $\kappa>1$ such that
%\begin{itemize}
%\item[i)]
$$M= B_1\cup \dots \cup B_k, \quad \text{and} \quad
%\item[ii)]  $
m(Dh_i^{-1}(x))>\kappa, \ \ \text{for all $x \in B_i$ and
$i=1,\dots,k$.}
$$
%\end{itemize}
\end{lemma}
As an immediate consequence of the lemma above one has the following corollary.
\begin{corollary}
\label{rem:exp-open} The set of expanding (finitely generated)
semigroup actions of $C^1$ diffeomorphisms of a compact
manifold $M$ is open.
\end{corollary}
We prove Theorem~\ref{thmA-generalize} using the main idea of Lemma~10.2 in \cite{BBD12}.
\begin{proof}[Proof of Theorem~\ref{thmA-generalize}]
%Consider an expanding minimal action of a se\-mi\-group $\Gamma$
%of $C^1$-diffeo\-mor\-phisms of a compact manifold $M$.
Lemma~\ref{lem1-cpt} provides a finite open cover $\{B_1,\dots,
B_k\}$ of $M$, a constant $\kappa>1$, a finite subset
$\mathcal{G}_1$ of $\mathcal{F}$ and elements $h_1,\dots,h_k$ in the semigroup generated by $\mathcal{G}_1$ such
that
%$h_i^{-1}: h_i(B_i) \to B_i$ and
$$
   d(h_i(x),h_i(y))< \kappa^{-1} d(x,y), \quad \text{for all $x,y \in h_i^{-1}(B_i)$ and
$i=1,\dots,k$.}
$$
Let $L>0$ be a Lebesgue number of cover $\{B_1,\dots,B_k\}$. Take $0<\varepsilon\leq L/2$. Since $M$ is a compact
and locally connected, every open ball of a uniform small size is connected. Hence, we also ask
$\varepsilon>0$ small enough so that every open ball of radius
$r\leq \varepsilon$ is a connected set. Since the $\Gamma$-orbit
of any point $x$ is dense, one can choose a finite $\varepsilon/4$-dense set of points
in the $\Gamma$-orbit of $x$, say $G_x$. Since  $G_x$ Consists of
finitely many elements of $\Gamma$, one can find a finite set
$\mathcal{G}_x\subset \Gamma$ in such a way that the orbit, under the action of semigroup generated by $\mathcal{G}_x$,
of any sufficintly close point to $x$ is also $\varepsilon/2$-dense. Using a compactness
argument, one can find a finite subset $\mathcal{G}_2\in\Gamma$
such that the orbit, under the action of semigroup  generated by $\mathcal{G}_2$, is $\varepsilon/2$-dense.
Now, we consider the sub-semigroup $\Upsilon$ of $\Gamma$ generated by
$\mathcal{G}=\mathcal{G}_1\cup \mathcal{G}_2$ and prove the $C^1$-robust minimality of its action To
do this, let us to take any semigroup $\tilde{\Upsilon}$ generated
by sufficiently small $C^1$-perturbation of the generators
$\mathcal{G}$ of $\Upsilon$ such that the following properties
hold:
\begin{itemize}
\item every point in $M$ has $\varepsilon$-dense
$\tilde{\Upsilon}$-orbit, and
\item there are a finite open cover
$\{B_1,\dots,B_k\}$ of $M$ with a Lebesgue number greater than $L/2$ and
maps $\{h_1,\dots,h_k\}$ in $\tilde{\Upsilon}$ such that the restriction of any $h_i$
to $h_i^{-1}(B_i)$ is a contraction of rate
$\kappa^{-1}$. %for all $i=1,\dots,k$.
\end{itemize}
%Now, a priori, we lose the density of the orbits, however, we just
%have $\varepsilon$-density. Shrinking, if necessary, the size of
%the neighborhood of perturbations, we assume that $\varepsilon\leq
%L/2$. Since $M$ is a compact manifold, it is locally connected and
%then every uniformly small enough open ball is a connected set.
%Hence, we also ask $\varepsilon>0$ small enough so that every open
%ball of radius $r\leq \varepsilon$ is a connected set.
%
Take $y\in B_i$ and $r>0$ such that $B(y,r)\subset B_i$ with
$r\leq \kappa^{-1}\varepsilon$. We claim that
$$h_i(B(h_i^{-1}(y),\kappa r)) \subset B(y,r).$$ Indeed, since the
restriction of $h_i$ to $h_i^{-1}(B_i)$ is a contraction of rate
$\kappa^{-1}$, one has that
$h_i(B(h^{-1}_i(y),\kappa r)) \subset B(y,r) \cup
   (M\setminus B_i).
$
By the inequality $\kappa r \leq \varepsilon$,  $B(h^{-1}_i(y),\kappa r)$ is
connected and the claim is deduced.

Finally, for a given point $x\in M$, we prove the density of the
$\tilde{\Upsilon}$-orbit of $x$. Take an arbitrary point $y\in M$ and choose $1\leq i\leq k$ in such a way that
$B(y,\kappa^{-1}\varepsilon)\subset B_i$.
By $\varepsilon$-density of the $\tilde{\Upsilon}$-orbit of $x$, there is
$h\in\tilde{\Upsilon}$ such that $h(x)\in
B(h^{-1}_i(y),\varepsilon)$. The claim above implies that
$h_i\circ h(x)\in B(y,\kappa^{-1}\varepsilon)$. This shows $\kappa^{-1}\varepsilon$-density
of the $\tilde{\Upsilon}$-orbit of $x$. Processing by induction, one can get $\kappa^{-n}\varepsilon$-density
of the $\tilde{\Upsilon}$-orbit of $x$, for any $n\in\mathbb{N}$, which proves the density of  the $\tilde{\Upsilon}$-orbit of $x$.
\end{proof}
%\begin{lemma}
%Let $\mathcal{F}$ be a family of homeomorphism of $M$ and take $x\in M$. Then the following conditions are equivalent:
%\begin{itemize}
%\item for all $U$ neighborhood of $x$, there exists $T_1,\ldots,T_n \in \G(\mathcal{F})$ (resp. in $\s(\mathcal{F})$) such that
%$$
%    \overline{\mathrm{Orb}(x)} \subset \bigcup_{i=1}^n T_i^{-1}(U)
%$$
%\item every point in $\overline{\mathrm{Orb}(x)}$ has dense orbit.
%\end{itemize}
%
%\end{lemma}
%\subsection{Some consequences of minimality and expansivity}

\section{Ergodicity of expanding actions}
\label{Sec-ergodicity}
\subsection{Proof of Theorem~\ref{thmB}}
\label{ss:proofB}
%the following lemma shows we give an equivalent
%more precisely the notions of minimality and ergodicity for group and semigroup actions.
The notion of expanding action of diffeomorphisms allows us to
adapt the Sullivan's exponential expansion strategy in~\cite{SS85}
to prove Theorem~\ref{thmB} following the spirit of~\cite{Na04}.
To accomplish this task, some preliminary lemmas are required.
\begin{lemma}
\label{preliminar} An action of a semigroup $\Gamma$ of continuous
maps on a compact space $X$ is minimal if and only if for every
open set $U \subset X$, there exist $T_1,\dots,T_m \in \Gamma$
such that $
    X= T^{-1}_1(U) \cup \dots \cup T^{-1}_m(U)$.
\end{lemma}
%\begin{proof}
%From the minimality, given any open set $U \subset M$, for every
%$y \in M$ there exists $T_y \in \Gamma$ such that  $T_y(y) \in U$.
%Now, using that $M$ is compact, we get a finite set of maps
%$T_1,\dots,T_m$ such that $M=T_1^{-1}(U)\cup\dots\cup
%T_m^{-1}(U)$. Reciprocally, consider an open set $U$ and a point
%$x\in M$. By the assumption there exist $T_1,\dots,T_m \in \Gamma$
%such that $M= T^{-1}_1(U) \cup \dots \cup T^{-1}_m(U)$. Thus,
%$T_i(x) \in U$, for some $1\leq i\leq m$. That is, $\Gamma$ acts
%minimally on $M$. %
%\end{proof}

Consider a family of maps $ \mathcal{H}=\{h_1,\ldots,h_k\} $.  Given $n\in\mathbb{N}$ and a sequence
$\omega=\omega_1\omega_2\dots\in\Omega=\{1,\dots,k\}^\mathbb{N}$
we denote $h_{\omega}^n=h_{\omega_n}\circ\dots \circ
h_{\omega_1}$. We also denote by $\mathrm{diaminf}\, A$ and
$\mathrm{diamsup}\, A$ the diameter of the largest ball inside of
$A$ and the smallest ball containing $A$ respectively.

\begin{lemma}[Bounded Distortion] \label{prop-dist-2}
Consider a family $\mathcal{H}=\{h_1,\ldots,h_k\} $ of
$C^{1+\alpha}$ diffeomorphisms of a compact Riemannian manifold $M$ and an open covering
$M=B_1\cup \dots \cup B_k$ in such a way that $m(Dh_j(x))>1$, for
all $x\in B_j$. Then,
\begin{enumerate}
\item there exists $L_{\mathcal{H}}>0$ such that for every
$n\in\mathbb{N}$,
$$
L_{\mathcal{H}}^{-1}<\bigg|\frac{\det Dh^n_\omega(x)}{\det
Dh^n_\omega(y)}\bigg|<L_{\mathcal{H}} \quad \text{for all $x,y\in M$}
$$
and any $\omega=\omega_1\omega_2\dots\in\Omega$
for which $
h^i_\omega(x), h^i_\omega(y) \in B_{\omega_{i+1}}$,
for $0\leq i <n-1$. \\[-0.3cm]
\item if, in addition, $h_j$ are  conformal
maps, there exists $K_\mathcal{H}>0$ such that for every $n\in
\mathbb{N}$,
%\begin{equation*}
$
\mathrm{diaminf}\, h_\omega^n(B) \geq K_\mathcal{H} \cdot
\mathrm{diamsup}\, h_\omega^n(B)$
%\end{equation*}
for all open ball $B \subset M$  and
$\omega=\omega_1\omega_2\dots\in\Omega$ for which
$
h^i_\omega(B) \subset B_{\omega_{i+1}}$, for $0\leq i <n-1$.
\end{enumerate}
\end{lemma}
\begin{proof}
Put $\digamma_j=\log |\det Dh_j|$. By the assumption, $\digamma_j$
is $\alpha$-H\"older and so there is a constant $C>0$ such that
for any $x,y$, $|\digamma_j(x)-\digamma_j(y)|\leq C
d(x,y)^\alpha$. Fixing $n\in\mathbb{N}$, let $x$, $y$ be two
points in $M$ and $\omega$ a sequence in $\Omega$ for which
$h^i_\omega(x), h^i_\omega(y) \in B_{\omega_{i+1}}$, for $0\leq i
<n-1$. Since $h_j$ is an expanding map on $B_j$, there exists a
constant $\kappa>1$ such that $\kappa d(x,y)\leq
d(h_j(x),h_j(y))$, for all $x,y \in B_j$ and $j=1,\dots,k$. Thus,
$$
 \kappa \,d(h^{i}_\omega(x),h^{i}_\omega(y)) \leq d(h^{i+1}_\omega(x),h^{i+1}_\omega(y))
$$
and hence,
$$
d(h^i_\omega(x),h^i_\omega(y)) \leq \kappa^{-(n-i)}
d(h_\omega^n(x),h_\omega^n(y))\leq  K \kappa^{-(n-i)},
$$
for all $i=0,\dots,n-1$ where $K=\max_{j=1,\dots, k}
\textrm{diam}(B_j)$. This implies that
\begin{align*}
\log\frac{|\det Dh^n_\omega(x)|}{|\det Dh^n_\omega(y)|}
&=\sum_{i=0}^{n-1}|\digamma_{\omega_{i+1}}(h^i_\omega(x))-\digamma_{\omega_{i+1}}(h^i_\omega(y))|
\leq C \sum_{i=0}^{n-1} d(h^i_\omega(x),h^i_\omega(y))^\alpha \\
&\leq C \sum_{i=0}^{n-1}(K\kappa^{-(n-i)} )^\alpha \leq C
K^\alpha\sum_{i=0}^{\infty}\kappa^{-i\alpha}.
\end{align*}
Taking $L_{\mathcal{H}}=\exp\{CK^\alpha\kappa^{-\alpha}/
(1-\kappa^{-\alpha}) \}$, the desired first inequality holds. To
prove the second inequality, first, note that
$$
  2r \cdot \inf_{x\in B} m(Dh^n_\omega(x)) \leq \mathrm{diaminf}\, h^n_\omega(B)
  \leq \mathrm{diamsup}\, h^n_\omega(B) \leq 2r \cdot \sup_{x\in B}
  \|Dh^n_\omega(x)\|,
$$
where $r>0$ is the radius of the open ball $B$. On the other hand,
from the conformality of $h_j$
$$
    \|Dh_\omega^n(x)\|=m(Dh^n_\omega(x)) =
    \prod_{i=0}^{n-1} a_{\omega_{i+1}}(h^i_\omega(x)) \eqdef a_\omega^n(x).
$$
Hence,
\begin{equation}
\label{eq:inequality}
    \frac{\mathrm{diamsup}\, h^n_\omega(B)}{\mathrm{diaminf}\,
    h^n_\omega(B)} \leq \frac{\sup\{a_\omega^n(x) \, : \, x\in
    B\}}{\inf\{a^n_\omega(x)\, : \, x\in B\}}.
\end{equation}
Since $\digamma_j=\log |a_j|$ is $\alpha$-H\"older and
$|a_j(x)|>1$, for all $x\in B_j$, by means of an analogous
distortion argument as above, one can get similarly a constant
$\tilde{K}_\mathcal{H}>0$ such that
\begin{equation}
\label{eq:tilde}
   \log \frac{|a^n_\omega(x)|}{|a^n_\omega(y)|} \leq
   \tilde{K}_\mathcal{H},
\end{equation}
for any two points $x, y$ and sequence $\omega$ for which
$h^i_\omega(x), h^i_\omega(y) \in B_{\omega_{i+1}}$, for $0\leq i
<n-1$. In view of the inclusion $h^i_\omega(B) \subset
B_{\omega_{i+1}}$, for $0\leq i <n-1$, and combining
\eqref{eq:inequality} and \eqref{eq:tilde} one can get the second
inequality for $K_\mathcal{H}=\exp(\tilde{K}_\mathcal{H})>0$.
\end{proof}

A point $x\in M$  is a \emph{Lebesgue density point} of a
measurable set $A\subset M$ if
$$\lim_{r\rightarrow 0}\frac{\lambda(A\cap
B(x,r))}{\lambda(B(x,r))}=1,$$ where $\lambda$ is the normalized
Lebesgue measure of $M$. Denote by $DP(A)$ the set of density
points of $A$.
%By Lebesgue Density Theorem, $\lambda$-almost every
%point in $A$ is a density point. That is, $\lambda(A\setminus
%DP(A))=0$.
We use the notation $E \mathring{\subset}\,F$, and say that {\it
$E$ is contained (\emph{mod~0}) in $F$}, if $\lambda(E\setminus F)
=0$.
%%In particular, if $E$ is open and mod 0 contained in $F$,
%%then $E\subset \overline{F}$.

The following proposition is the main key ingredient to prove
Theorem~\ref{thmB}.
\begin{proposition}
\label{prop:bola} Consider an expanding action of a semigroup
$\Gamma$ of $C^{1+\alpha}$ conformal diffeomorphisms of a compact Riemannian
manifold $M$. Then, there exists $r>0$ such that for every
$\Gamma$-invariant set $A\subset M$ whose complementary $A^c\eqdef
M\setminus A$ has positive Lebesgue measure there exist an open
ball $B$ of radius $r>0$ so that $\lambda(B\setminus A^c) =0$,
i.e., $B\mathring{\subset}\, A^c$.
\end{proposition}
\begin{proof}
By Lemma~\ref{lem1-cpt}, there are maps $g_1,\dots,g_k$ in
$\Gamma$, open sets $B_1,\dots,B_k$ of $M$ and a constant
$\kappa>1$ such that $M= B_1\cup \dots \cup B_k$ and
$m(Dg_i^{-1}(x))>\kappa$, for all $x\in B_i$ and $i=1,\dots,k$. To
simplify the notation take $h_i=g_i^{-1}$. Let $L>0$ be a
Lebesgue number of the open cover above. The Lebesgue density
theorem allows us to take a density point $x_0$ of $A^c=M\setminus
A$. For every $0<\delta<L/2$, open ball $B(x_0,\delta)$ is
contained in some
$B_i$. % and hence $h_i(B(x_0,\delta))$.
Moreover, by the expanding property of $h_i$ on $B_i$, the set
$h_i(B(x_0,\delta))$ contains an open ball of radius $\kappa
\delta$ centered at $h_i(x_0)$. If
$\mathrm{diamsup}\,h_i(B(x_0,\delta))<L$ then $h_i(B(x_0,\delta))$
is again contained in some element of the cover, say $B_j$.
Repeating the process one can show that $h_j\circ
h_i(B(x_0,\delta))$ contains an open ball of radius
$\kappa^{2}\delta$ centered at $h_j\circ h_i(x_0)$. Since this
process provides open balls of strictly increasing radius, one
gets $n\in \mathbb{N}$ and $\omega_j \in \{1,\dots,k\}$, for
$j=1,\dots,n$, such that for $h=h_{\omega_{n}}\circ\dots\circ
h_{\omega_1}$ the followings hold:
$$
h(B(x_0,\delta))\subset B_{\omega_{i+1}},~~~~\text{for $0\leq
i<n$},~~~~\text{and}~~~~ \mathrm{diamsup}\,h(B(x_0,\delta))\geq L.
$$
According to the second item in Lemma~\ref{prop-dist-2},
$h(B(x_0,\delta))$ contains a ball of radius $LK_\mathcal{H}/2$.
Therefore, there exists a sequence $\omega \in
\Omega=\{1,\dots,k\}^\mathbb{N}$ such that for each $t\in
\mathbb{N}$, taking $\delta_t=L/4t$, there exists
$n=n(t)\in\mathbb{N}$ and $x_t\in M$ such that
\begin{equation}
\label{eq:22} B(x_t,LK_\mathcal{H}/2) \subset
h^n_\omega(B(x_0,\delta_t)).
\end{equation}
By the compactness of $M$, taking a subsequence if necessary,
$x_t$ converges to some point, say $x$. Hence,  %since $x_n$ is uniformly far from the boundary then $x\in B$. Moreover,
there exists $t_0 \in \mathbb{N}$ such that
\begin{equation}
\label{eq:33} B(x,r) \subset B(x_t,LK_\mathcal{H} /2), \ \
\text{for all $t\geq t_0$},
\end{equation}
where $r=LK_\mathcal{H}/4>0$. On the other hand, the
inclusions~\eqref{eq:22} and~\eqref{eq:33}, and the (forward)
$\Gamma$-invariance of $A$ implies that $g^{-1}(A^c) \subset A^c$,
for all $g\in \Gamma$. Hence, for every $t\geq t_0$,
$$
 h^n_\omega(B(x_0,\delta_t)\setminus A^c)
 \supset h^n_\omega(B(x_0,\delta_t))\setminus A^c
 \supset B(x,r)\setminus A^c.
$$
For every $t\geq t_0$ one has that,
\begin{equation}
\label{eq:44}
  \frac{\lambda(B(x,r)\setminus A^c)}{\lambda(M)}
  \leq \frac{\lambda (h^n_\omega(B(x_0,\delta_t)\setminus A^c)
  )}{\lambda(h^n_\omega(B(x_0,\delta_t)))}
  \leq L_\mathcal{H}\, \frac{\lambda(B(x_0,\delta_t)\setminus
A^c)}{\lambda(B(x_0,\delta_t))}.
\end{equation}
The last inequality is implied by the bounded distortion property,
Lemma~\ref{prop-dist-2}. Indeed, by the construction, for every $z
\in B(x_0,\delta_t)$ one has that  $h^{i}_\omega(z) \in
B_{\omega_{i+1}}$, for $0\leq i<n$. Now, it suffices to note that
\begin{align*}
\lambda(h^n_\omega(B(x_0,\delta_t)\setminus A^c))
&=\int_{B(x_0,\delta_t)\setminus A^c} |\det Dh^n_\omega | \
d\lambda  \\
&\leq \lambda(B(x_0,\delta_t) \setminus A^c) \sup_{z\in
B(x_0,\delta_t)}
|\det Dh^n_\omega(z)|, \\
\lambda(h^n_\omega(B(x_0,\delta_t))) &=\int_{B(x_0,\delta_t)}
|\det Dh^n_\omega|\ d\lambda  \\
&\geq \lambda(B(x_0,\delta_t))
\inf_{z\in B(x_0,\delta_t)} |\det Dh^n_\omega(z)|
\end{align*}
and therefore, Lemma~\ref{prop-dist-2} implies that
\begin{align*}
\frac{\lambda(h^n_\omega(B(x_0,\delta_t)\setminus A^c)
)}{\lambda(h^n_\omega(B(x_0,\delta_t)))} \leq L_\mathcal{H} \,
\frac{\lambda(B(x_0,\delta_t) \setminus A^c)
}{\lambda(B(x_0,\delta_t)) }.
%\frac{\sup_{z\in B(x_0,\delta_n)}
%|\det Dh^n_\omega(z)|}{\inf_{z\in B(x_0,\delta_n)}
%|\det Dh^n_\omega(z)|}.
\end{align*}
Since $x_0\in DP(A^c)$, one gets that
$$
\lim_{t\to\infty} \lambda(B(x_0,\delta_t)\setminus
A^c)/\lambda(B(x_0,\delta_t))=0.
$$
Now, inequality~\eqref{eq:44} implies that
$\lambda(B(x,r)\setminus A^c)=0$ and the proof is complete.
%we get that
%$$
%\frac{\lambda( f_n(\hat{\Theta}\cap A_n))}{\lambda(f_n(A_n))} \geq
%L_{\mathcal{H}}^{-1}\,\frac{\lambda(\hat{\Theta}\cap A_n)}{\lambda(A_n)}.
%$$
%
\end{proof}

Now, we are ready to prove Theorem~\ref{thmB}.
%\begin{theopargself}
\begin{proof}[Proof of  Theorem~\ref{thmB}]
Consider an expanding minimal action of a semigroup
$\Gamma$ %generated by a family (non-necessarily finite) $\mathcal{F}$
of $C^{1+\alpha}$ conformal diffeomorphisms of a compact manifold
$M$. Let $A$ be a (forward) $\Gamma$-invariant measurable set.
Observe that $A^c=M\setminus A$ is also an invariant set for the
backward semigroup action, i.e., for the action of the semigroup
generated by the inverse maps of $\Gamma$. Supposing
$\lambda(A^c)>0$, we prove $\lambda(A^c)=1$.

By Proposition~\ref{prop:bola}, there exists $x\in M$ and $r>0$
such that $B(x,r) \, \mathring \subset \, A^c$. In view of the
minimality of the action, Lemma~\ref{preliminar} provides
$T_1,\dots, T_m \in \Gamma$ in such a way that
$$M= T_1^{-1}(B(x,r))
\cup \dots\cup T_m^{-1}(B(x,r)).
$$
Since $A^c$ is a forward invariant for the backward semigroup
action, one has that
$$
T_i^{-1}(B(x,r))\setminus A^c \subset T_i^{-1}(B(x,r)\setminus
A^c),$$ for all $i=1,\dots,m$. Hence, the quasi-invariance of
$\lambda$ for $C^1$-diffeomorphisms implies that
$\lambda(T_i^{-1}(B(x,r))\setminus A^c)=0$ and so $\lambda(M
\setminus A^c ) =0$. This proves that $\lambda(A^c)=1$ concluding
the proof of the theorem.
\end{proof}
%\end{theopargself}

%$C^1$-diffeomorphisms we get that the following remark.
%\begin{remark}
%Every expanding minimal semigroup action of
%$C^{1+\alpha}$-diffeo\-mor\-phisms of a compact manifold is also
%backward ergodic with respect to the Lebesgue measure.
%\end{remark}
%\subsection{Robustness of minimality and ergodicity}
%
%Theorem~\ref{thmA} implies that a group action of $C^{1+\alpha}$-diffeomorphisms of a compact manifold is Lebesgue-ergodic if it is expanding and minimal. Notice that every small enough $C^0$-perturbation of a expanding action it is also expanding. Thus, in order to get the robustness of the ergodicity (Corollary~\ref{thmB} and \ref{thmC}) we need to the following result (compared with~\cite[Lemma~10.2]{BBD}).
%\begin{theorem}
%\label{thm:robust-minimality}
%Every expanding minimal group action is $C^0$-robustly minimal.
%\end{theorem}
%\begin{proof}
%
% \end{proof}
%\begin{proof}[Proof of Corollary~\ref{thmB}]
%By Lemma~\ref{thm:robust-minimality} every small enough $C^0$-perturbation
%of $\G(\mathcal{F})$ is also an expanding minimal group. Thus, as a consequence of Theorem~\ref{thmA}, every group generated by $C^{1+\alpha}$-diffeomorphism in this $C^0$-neighborhood is ergodic with respect to the Lebesgue measure and we conclude the robustness of the ergodicity.
% \end{proof}
%
%
%\begin{proof}[Proof of Corollary~\ref{thmC}]
%
% \end{proof}
%\subsection{Proof of Corollary\ref{corB}}

\section{Contracting iterated function systems}
In this section we study contracting iteration function systems
(IFSs for short), i.e, semigroup actions generated by finitely
many contracting maps. These actions have a unique compact minimal
(self-similar) IFS-invariant set called \emph{Hutchinson
attractor}. The restriction of IFS to this attractor is minimal
and expanding. We first study the local ergodicity on Hutchinson
attractor of contracting IFS (see Theorem~\ref{pro5}). Afterward,
we apply this local ergodicity to show the equivalence between the
(unique) stationary measure of contracting IFSs and the
Lebesgue measure (Proposition~\ref{pro3}).

\subsection{Contracting iterated function systems} \label{vitali}
Recall that by an \emph{iterated function system} (IFS) we
understand the action of a semigroup generated by a finite family
of continuous maps on a metric space. IFS is said to be
\emph{contracting} if its generators are contracting maps. A
crucial fact about the contracting IFS is that it has unique
compact set $\Delta$, called \emph{Hutchinson attractor},
satisfying
\begin{equation}
\label{hutchinson}
 \Delta=\bigcup_{i=1}^k
h_i(\Delta),
\end{equation}
where $h_1,\dots, h_k$ are the generators \cite{Hut81}. Denote by
$\IFS(\mathcal{H})$ the semigroup generated by a family
$\mathcal{H}=\{h_1,\dots,h_k\}$ of $C^1$ contracting maps of a
manifold $M$. Clearly, the action of $\IFS(\mathcal{H})$ on the
Hutchinson attractor $\Delta$ is minimal and expanding.

In what follows, we work with the contacting IFS whose Hutchinson
attractor $\Delta$ has positive Lebesgue measure. For instance,
this is the case of a contracting IFS realized by a blending
region (see Subsection~\ref{blending}).

\begin{definition}
A contracting IFS is \emph{ergodic on the Hutchinson
attractor} $\Delta$ (with respect to $\lambda$) if
$\lambda(A)\in\{0,\lambda(\Delta)\}$, for all IFS-invariant
$\lambda$-measurable set $A \subset \Delta$.
\end{definition}

In order to study the local ergodicity with respect to
Lebesgue measure $\lambda$ we need to impose a regularity
criterion.

\begin{definition} We say that $\IFS(\mathcal{H})$ is \emph{Vitali-regular
(V-regular)} if there is a Vitali-regular cover
$$
\mathcal{V}\subset \{h(\Delta): \,h\in \IFS(\mathcal{H}) \}
$$
of its Hutchinson attractor $\Delta$. This means that there is a
constant $C>0$ such that
\begin{itemize}
\item for any $V\in\mathcal{V}$,
$(\diam V)^d\leq C\lambda(V)$,  where $d=\dim M$,
\item for any $\delta>0$ and $x\in\Delta$, there is
$V\in\mathcal{V}$, with $x\in V$ and $\diam V\leq\delta$.
\end{itemize}
We say that $\IFS(\mathcal{H})$ has \emph{bounded distortion (BD)}
if there is $L>0$ such that for every $h \in \IFS(\mathcal{H})$,
$$
L^{-1}<\bigg|\frac{\det Dh(x)}{\det Dh(y)}\bigg|<L, \quad
\text{for any $x,y\in \Delta$.}
$$
\noindent Finally, a contracting IFS is {\it BDV-regular} if it is
Vitali-regular and has bounded distortion.
\end{definition}
It is not hard to see that the conformality of the elements of
$\mathcal{H}$ implies Vitali-regularity of $\IFS(\mathcal{H})$. On
the other hand, there are two classical tools to guarantee the
bounded distortion property. One is $C^{1+\alpha}$-regularity of
the generators and the other one, which is actually weaker, is
Dini-regularity of the generators. Recall that a $C^1$-map $\phi$
is \emph{Dini} if,
$$
   \int_{0}^1 \frac{\Omega(\log \|D\phi(\cdot)\|,t)}{t} \,
   dt<\infty,
   %\quad \text{for all $x\in D$}
$$
where $\Omega(p,t)$ is the modulus of continuity of $p:M\to
\mathbb{R}$. If the generators are $C^{1+\alpha}$ the bounded
distortion follows arguing as in Lemma~\ref{prop-dist-2}. By means
of similar arguments one can also prove the bounded distortion
property for Dini-contracting IFS (see~\cite{FL99}).
\subsection{Local ergodicity on Hutchinson attractor}
\label{Hut} %We begin by the definition of the local ergodicity we
%are using here. In what follows we will assume that we work with
%contacting IFS whose Hutchinson attractor $\Delta$ has positive
%Lebesgue measure. For instance, this is the case of a contracting
%IFS generated by blending region (see the following subsection).
%
%A contracting IFS is said to be \emph{ergodic on the Hutchinson
%attractor} if $\lambda(A)\in\{0,\lambda(\Delta)\}$, for all
%IFS-invariant $\mu$-measurable set $A \subset \Delta$.
The main
result of this subsection is the following.
\begin{theorem}
\label{pro5} Every BDV-regular contracting IFS on the Hutchinson
attractor is ergodic with respect to the Lebesgue measure.
\end{theorem}
We provide the proof of Theorem~\ref{pro5} through a few lemmas.
%We continue the discussion by presenting two basic lemmas for the
%proof of Theorem~\ref{thmD}.
%%%%%%%%%% corte aqui

\begin{lemma}
Every Vitali-regular contracting IFS satisfies the following
property: for any $x\in DP(\Delta)$ there are two constants
$C_1,C_2>0$ such that for every $\delta>0$ there is
$\mathcal{V}_\delta \subset \{h(\Delta): h\in
\IFS(\mathcal{H})\}$ with %having the following property
\begin{equation}
\label{eq:prop1}
   C_1 \lambda(B(x,\delta)) \geq \lambda(\bigcup_{V\in \mathcal{V}_\delta}V)=
   \sum_{V\in \mathcal{V}_\delta} \lambda(V) \geq C_2 \lambda( B(x,\delta)\cap \Delta).
\end{equation}
\end{lemma}
\begin{proof}
Fix $x\in \Delta$ and $\delta>0$. Since the contracting IFS is
Vitali-regular, by Vitali's covering theorem for the Lebesgue
measure, there exists a finite or countably infinite disjoint
subcollection $\mathcal{V}_\delta$ of $\{h(\Delta): h \in
\IFS(\mathcal{H})\}$ such that
$$
B(x,\delta) \cap \Delta \mathring{=} \biguplus_{V\in
\mathcal{V}_\delta} V.
$$
This implies~\eqref{eq:prop1} taking $C_1=C_2=1$.  \end{proof}

As before, we denote by $\IFS(\mathcal{H})$ the semigroup
generated by the family of contracting maps
$\mathcal{H}=\{h_1,\dots,h_k\}$. Given a set $A \subset M$ put
$$
  \mathrm{Orb}^-_\mathcal{H}(A) = \bigcup_{h\in \IFS(\mathcal{H})} h^{-1}(A) \cup A.
$$
\begin{lemma}
\label{lem:key}Assume that $\IFS(\mathcal{H})$ satisfies
~\eqref{eq:prop1} and $A$ is a measurable set of $M$. Then
$$
  \lambda(\mathrm{Orb}_\mathcal{H}^-(A)\cap \Delta ) \in \{0, \lambda(\Delta)\}.
$$
Moreover, if $DP(A)\cap DP(\Delta) \not=\emptyset$ then it always
holds that $ \Delta \,\mathring\subset \,
\mathrm{Orb}_\mathcal{H}^-(A). $
\end{lemma}
\begin{proof}
%Let $A$ be a measurable set in the Hutchinson attractor $\Delta$ such that $h^{-1}(A) \subset A$ for all $h$ in the IFS.
Put $\Theta=\mathrm{Orb}_\mathcal{H}^-(A)$. By the assumption
$\lambda(\Delta)>0$ and so one can assume that $\lambda(\Theta\cap
\Delta )>0$. We prove $ \lambda(\Delta\setminus\Theta)=0$.

The Lebesgue density theorem implies the existence of a density
point $x$ of $\Theta\cap \Delta$.
%In order to show this, we  assume that $\lambda(\Delta \setminus A)>0$ and then we will get a contradiction.
%We consider that $\lambda(A)>0$ and we prove that $\lambda(A)=\lambda(\Delta)$. In order to do this, we assume that $\lambda(\Delta \setminus A)>0$ we get a contradiction.
%By hypothesis, there exists a density point $x$ of both $A$ and $\Delta$.
Hence, one can find $\delta_0>0$ such that
$$
   \lambda(B(x,\delta)\cap \Delta)>
   \lambda(B(x,\delta)) /2, \quad \text{for all $\delta_0 \geq \delta >0$}.
$$
By assumption, there are two constants $C_1,C_2>0$ such that for
every $\delta >0$ with $\delta \leq \delta_0$, there is
$\mathcal{V}_\delta \subset \{h(\Delta): h\in \IFS(\mathcal{H})\}$
satisfying the following inequalities
\begin{equation*}
   C_1 \lambda(B(x,\delta)) \geq \lambda(\bigcup_{V\in \mathcal{V}_\delta}V)
   =\sum_{V\in \mathcal{V}_\delta} \lambda(V) \geq C_2 \lambda( B(x,\delta)\cap \Delta).
\end{equation*}
By the backward invariance of $\Theta$, i.e.
$h^{-1}(\Theta)\subset \Theta$, for all $h\in \IFS(\mathcal{H})$, %and by the bounded distortion
one gets that
\begin{align*}
%$$
\frac{\lambda(B(x,\delta)\setminus \Theta)}{\lambda(B(x,\delta))}
&\geq \frac{1}{C_1} \frac{\lambda(\cup_{V\in \mathcal{V}_\delta} (
V \setminus \Theta))}{\lambda(B(x,\delta))} \\ &= \frac{1}{C_1}
\sum_{V\in \mathcal{V}_\delta} \frac{\lambda(V\setminus \Theta)}{
\lambda(B(x,\delta))} %\\[0.2cm] &
\geq \frac{1}{C_1} \sum_{h(\Delta)\in \mathcal{V}_\delta}
\frac{\lambda(h(\Delta \setminus
\Theta))}{\lambda(B(x,\delta))}.%$$
\end{align*}
Since the IFS has bounded distortion (with constant $L>0$), it
follows that
\begin{align*}
\sum_{h(\Delta)\in \mathcal{V}_\delta} \frac{\lambda(h(\Delta
\setminus \Theta))}{\lambda(B(x,\delta))} &=
\sum_{h(\Delta)\in
\mathcal{V}_\delta} \frac{\lambda(h(\Delta \setminus
\Theta))}{\lambda(h(\Delta))}
\frac{\lambda(h(\Delta))}{\lambda(B(x,\delta))} \\
&\geq L \,
\frac{\lambda(\Delta \setminus \Theta)}{\lambda(\Delta)}
\sum_{V\in \mathcal{V}_\delta}
\frac{\lambda(V)}{\lambda(B(x,\delta))}
\\[0.2cm]  &\geq L C_2 \,
\frac{\lambda(\Delta \setminus \Theta)}{\lambda(\Delta)}
\frac{\lambda(B(x,\delta)\cap \Delta)}{\lambda(B(x,\delta))} \\[0.2cm]  &>
\frac{L C_2}{2} \, \frac{\lambda(\Delta \setminus
\Theta)}{\lambda(\Delta)}.
\end{align*}
Therefore, one obtains
$$
\frac{\lambda(B(x,\delta)\setminus \Theta)}{\lambda(B(x,\delta))}
\geq \frac{L}{2}\frac{C_2}{C_1} \frac{\lambda(\Delta\setminus
\Theta)}{\lambda(\Delta)}, \quad \text{for all $0<\delta \leq
\delta_0$.}
$$
%since $\lambda(\Delta\setminus \Theta)
Since $x$ is a Lebesgue density point of $\Theta$ one gets
$$
 0=\lim_{\delta\to 0} \frac{\lambda(B(x,\delta)\setminus \Theta)}
 {\lambda(B(x,\delta))} \geq \frac{L}{2}\frac{C_2}{C_1}
 \frac{\lambda(\Delta \setminus \Theta)}{\lambda(\Delta)}.
$$
This implies that $\lambda(\Delta \setminus \Theta)=0$ concluding
the proof of the first part of the lemma.
Now, to show the second part, take a density point $x$ of both $A$
and $\Delta$. Since $A\subset \Theta$, the
argument above  shows that
$$
   \frac{\lambda(B(x,\delta)\setminus A)}{\lambda(B(x,\delta))}
   \geq \frac{\lambda(B(x,\delta)\setminus \Theta)}
   {\lambda(B(x,\delta))} > \frac{L}{2}\frac{C_2}{C_1}
   \frac{\lambda(\Delta\setminus \Theta)}{\lambda(\Delta)}.
$$
As $x$ is a density point of $A$ one can get $\lambda(\Delta\setminus \Theta)=0$,
concluding the lemma.
\end{proof}

\begin{proof}[Proof of Theorem~\ref{pro5}]
Let $A$ be a measurable set containing in the Hutchinson attractor $\Delta$
such that $h(A)\subset A$, for all $h\in\IFS(\mathcal{H})$. Putting
$\hat A = \Delta \setminus A$, wee claim that
$$
  h^{-1}(\hat A\cap h(\Delta)) \subset \hat A, \quad \text{for all $h\in \IFS(\mathcal{H})$.}
$$
By contradiction, suppose $x$ is a point of $\hat A \cap
h(\Delta)$ such that $h^{-1}(x)\not\in \hat A$. Since $x\in\hat
A$, by the invariance, $x \not\in h(A)$, that is $h^{-1}(x)
\not\in A$. On the other hand,  by the assumption,
$h^{-1}(x)\not\in\hat A$. Hence $h^{-1}(x)\not\in \Delta$ which
yields a contradiction. Thus,
$$
   \mathrm{Orb}^-_\mathcal{H}(\hat A) \cap \Delta=
   \bigcup_{h\in\IFS(\mathcal{H})} h^{-1}(\hat{A}\cap h(\Delta)) \cup (\hat{A}\cap \Delta) = \hat{A}.
$$
Now, by Lemma~\ref{lem:key} one has that $\lambda(\hat{A}) \in
\{0,\lambda(\Delta)\}$, that means $\lambda(A)$ equals to
either zero or $\lambda(\Delta)$, concluding the proof.
\end{proof}
\begin{remark}
\label{rem1} Theorem~\ref{pro5} is also valid for the Hausdorff
$s$-dimensional measure. More details on the adaption of the proof above to such case can be found in~\cite{FL99,HLW02}.
\end{remark}
\subsection{Stationary measures} \label{stationary}
We apply the local ergodicity on the Hutchinson attractor %in the
%previous subsection
to study the important question concerning the
equivalence of the stationary measure and The Lebesgue measure
(\cite{MS98,LNR01,PSS00,PSS06}). Recall that a probability measure $\mu$ is {\it
stationary measure} for $\IFS(h_1,\dots,h_k)$ with positive
probabilities $p_1,\dots,p_k$ ($\sum_{i=1}^k p_i=1$) %and $\mathcal{H}=\{h_1,\dots,h_k\}$,
if
\begin{equation}
\label{SSA}
   \mu= \sum_{i=1}^k p_i \cdot (\mu\circ h_i^{-1}).
\end{equation}
It is well known that any contracting IFS admits a unique stationary measure (\cite{DF,Hut81}).
\begin{proposition}
\label{pro3} Suppose that $\mu$ is the stationary probability
measure on $M$ corresponding to a BDV-regular contracting IFS.
Let $\Delta$ be the support of $\mu$. If $\mu$ is
not singular to Lebesgue measure $\lambda$, then
\begin{enumerate}
\item $\mu$ is absolutely continuous with respect to $\lambda$,
\item $\lambda|_\Delta$ is absolutely continuous with respect to $\mu$.
\end{enumerate}
\end{proposition}
\begin{proof}
By \cite{Hut81} the support $\Delta$ of the unique stationary
measure $\mu$ is the Hutchinson attractor.
Using the Lebesgue decomposition one can split
$\mu=\mu_{ac}+\mu_{s}$ into an absolute part and a singular part.
Observe that both measures $\mu_{ac}$ and $\mu_s$ satisfy the
self-similarity relation like~(\ref{SSA}). Thus $\mu$ is either
singular or absolutely continuous with respect to Lebesgue measure.

To conclude the result, we will prove that if $\mu$ is not
singular with respect to Lebesgue measure $\lambda$  then
$\mu$ is equivalent to the restriction of $\lambda$ to the support
of $\mu$. Suppose that $\mu$ is absolutely continuous but not
equivalent to the restriction of the Lebesgue measure.
This implies that $\lambda(\Delta)>0$ and, simultaneously, the existence of a
measurable set $A\subset \Delta$ with $\lambda(A)>0$ and
$\mu(A)=0$. Let $$\Theta=\mathrm{Orb}^-_\mathcal{H}(A)\cap \Delta.$$
Since $\lambda(A)>0$, Lemma~\ref{lem:key} implies that
$\lambda(\Theta)=\lambda(\Delta)$. On the other hand, the
self-similarity of the measure implies that $\mu(h^{-1}(A))=0$,
for all $h\in\IFS(\mathcal{H})$ and hence $\mu(\Theta)=0$. This
means that $\mu$ is singular with respect to the restriction of the
Lebesgue measure, contradicting the absolute continuity of $\mu$.
\end{proof}
\begin{remark}
Again, following~\cite{HLW02}, this result may easily
be generalized to the case of a Hausdorff $s$-dimensional measure
and also to a probability measure $\mu$ satisfying eigen-equation
$$
  \lambda\mu = \sum_{i=1}^k p_i(\cdot)  (\mu \circ h_i^{-1}),
$$
for some $\lambda>0$, where $p_i(\cdot)$ is a family of
continuous probability functions on $M$.
\end{remark}
\section{Blending regions} \label{blending}
In this section we introduce the notion of blending region for
semigroup actions as local feature allowing us to provide a broad
class of minimal actions having the expanding property (see
Theorem~\ref{pro1}). In fact, this is perhaps the simplest way to
yield a $C^1$-robust minimal action. Blending regions are also
connected with Hutchinson attractor with not empty interior of
certain contracting IFSs. This connection allows to apply the
results of Section~\ref{vitali} to construct a class of actions,
not necessarily conformal, for which minimality
implies ergodicity, assuming some local extra regularity in
the blending region (see Theorem~\ref{pro2}).

\subsection{Blending regions}
As it is mentioned before the notion of blending region is the main tool
to produce robust minimal actions (\cite{BKR12,BR13,GHS10,HN13}). Following the same essential
strategies, we prove a slightly different result on robust
minimality (see Theorem~\ref{pro1} below). We begin by the formal
definition of the notion.

\begin{definition}
\label{def-blending}
 An open set $B \subset M$ is called
\emph{blending region} for a semigroup $\Gamma$ of diffeomorphisms
of $M$ if there exist $h_1,\dots, h_k\in \Gamma$ and an open set
$D\subset M$ such that $\overline{B}\subset D$ and
\begin{enumerate}
\item \label{covering-property} $\overline{B}\subset h_1(B)\cup \dots \cup h_k(B)$,
\item $h_i: \overline{D} \to D$ is a contracting map for $i=1,\dots,k$.
\end{enumerate}
The semigroup action generated by $\mathcal{H}=\{h_1,\dots,h_k\}$
on $\overline{D}$ is called \emph{associated contracting iterated
function system}. A blending region $B$ is said to be
\emph{globalized}  if there exist maps
$T_1,\dots,T_m,S_1,\dots,S_n \in \Gamma$ such that
%\begin{equation*}
$$M=T^{-1}_1(B)\cup \dots \cup T^{-1}_m(B)=S_1(B)\cup \dots \cup S_n(B).$$
\end{definition}

It is not difficult to see that if $\Gamma$ acts forward and
backward minimally on $M$ then any open subset of $M$ is
globalized for $\Gamma$. For instance, this is the case of a
minimal group action.

\subsection{Globalized blending regions, a criterium to yield robust minimality}
\label{ss:minimal-blender} Theorem~\ref{pro1}, below, provides
relatively broad class of examples of robustly minimal actions.
More precisely, it says that if the minimality of an action
realized by a blending region then the action is robustly minimal.
In view of Theorem~\ref{thmA}, the main point of the proof is that
the existence of a blending region for a minimal action converts
the action into an expanding one.

% As a corollary of the result
%above one gets that a blending region provides also the robustness
%of the ergodicity assuming global $C^{1+\alpha}$-regularity of the
%generators.
%By a \emph{forward} (resp.~\emph{backward}) minimal
%semigroup action we understand that the action is minimal (resp.~the
%action of the semigroup generated by the inverse maps is minimal).
\begin{theorem}
\label{pro1}
%\begin{corollary}
%\label{prop:minimalidad}
Consider a finitely generated semigroup $\Gamma$ %generated by a family
of $C^1$-diffeomorphisms of a compact manifold $M$.
%with a blending
%region
%$B$ having a forward cycle.
Assume that there exist a globalized open set $B\subset M$ for
$\Gamma$, an open set $D \subset M$ with $\overline{B}\subset D$
and maps $h_1,h_2,\dots \in \Gamma$ such that each $h_i:
\overline{D}\to D$ is a contraction of rate $\beta<1$ and
\begin{equation}
\label{cover00} B\subset \bigcup_{i=1}^\infty h_i(B).
\end{equation}
%and
%\item \label{ite:2} $M=T^{-1}_1(B)\cup \dots \cup T^{-1}_m(B)=S_1(B)\cup \dots \cup S_n(B).$
%\end{enumerate}
Then, the action of $\Gamma$ is expanding and
%. Further, if $B$ has a cycle then the action $\Gamma$ is
$C^1$-robustly minimal. %Moreover, if the generators of $\Gamma$
%has derivatives H\"older continuous then it is also robustly ergodic
%with respect to the Lebesgue measure.
\end{theorem}
\begin{remark}
In Definition~\ref{def-blending}, the covering
property~\eqref{covering-property} holds for the closure of $B$.
Roughly, the strength of this definition is the robustness of the
property under the perturbation of the generators. In
Theorem~\ref{pro1}, we weaken this covering property and deduce
again the robust minimality.
\end{remark}
Before proving the theorem we prove a basic lemma.
\begin{lemma}
\label{lem-densidad-A} Let $B$ be an open set satisfying the
covering property ~\eqref{cover00}. Then for every $x\in B$, there
is a sequence $(i_n)_{n\geq 1}$, $i_n \in \mathbb{N}$  such that
$$
  x= \lim_{n\to\infty} h_{i_1}\circ
  \dots \circ h_{i_n}(y),~~ \text{for all} \ y\in B.
$$
\end{lemma}
\begin{proof}
We define recursively the sequence $(i_n)_{n\geq 1}$: once we have
found an integer $i_n$ such that
$$x\in h_{i_1}\circ\cdots \circ h_{i_n}(B),$$
we deduce from (\ref{cover00}) that
$$x\in h_{i_1}\circ\cdots \circ h_{i_n}(B)\subset \bigcup_{i=1}^{+\infty}
h_{i_1}\circ\cdots \circ h_{i_n}\circ h_i(B).$$ So, we can
find $i_{n+1}$ so that
$x\in h_{i_1}\circ\cdots \circ h_{i_n}\circ h_{i_{n+1}}(B).$
In such a way we have constructed a sequence of integers $(i_n)_{n\geq 1}$
such that
$$
x\in \bigcap_{n\geq 1} h_{i_1}\circ \dots \circ h_{i_n}(B).
$$
Moreover, since each $h_i$ is a contraction in $D$ of rate
$0<\beta<1$, it follows that
$$
\mathrm{diam}(h_{i_1}\circ \dots \circ h_{i_n}(B))\leq \beta^n
\mathrm{diam}(B).
$$
Since $x\in h_{i_1}\circ\cdots \circ h_{i_n}(B)$, we deduce that
for a given $y\in B$,
$$
d(h_{i_1} \circ \dots \circ h_{i_n}(y), x) \leq
\mathrm{diam}(h_{i_1}\circ \dots \circ h_{i_n}(B)) \leq  \beta^n
\mathrm{diam}(B)m
$$
for every $n\in \mathbb{N}$. Since $0<\beta<1$  then
$h_{i_1} \circ \dots \circ h_{i_n}(y)$  converges to $x$.
%$$x=\displaystyle \lim_{n\to \infty} h_{i_1} \circ
%\dots \circ h_{i_n}(y)
%$$
%and conclude proof. %the lemma is complete.
\end{proof}

\begin{proof}[Proof of Theorem~\ref{pro1}] Since
Given a point $x\in M$, in view of covering $M=S_1(B)\cup\cdots\cup S_n(B)$, there
is $i\in \{1,\dots, n\}$  such that $S_i^{-1}(x) \in B$. The
covering property~\eqref{cover00} allows us to iterate $S_i^{-1}(x)$ by some
$h_i^{-1}$ remaining in $B$. Since
$m(Dh_i^{-1}(z))\geq \beta^{-1}$, for all $z\in h_i(B)$ and
$i=1,\dots,k$, repeating this argument one gets $g\in\Gamma$ such
that $m(Dg^{-1}(x))>1$.
%That is, the action of $\Gamma$ on $M$ is expanding.

Given an open set $U$ and a point $x$ in $M$, by globalization
property, one gets
$$
M=T_1^{-1}(B)\cup\cdots\cup T_m^{-1}(B)=S_1(B)\cup\cdots S_n(B),
$$
one can find $i$ and $k$ in such a way that $T_i(x)\in B$ and $B\cap
S_k^{-1}(U)$ contains an open set. By Lemma~\ref{lem-densidad-A},
any point in $B$ has dense orbit in $B$, and so, there exists
$h\in \Gamma$ such that $h\circ T_i (x)\in S_k^{-1}(U)$ or
$S_k\circ h \circ T_i(x)\in U$. This shows the minimality of the
action. Now, Theorem~\ref{thmA} implies $C^1$-robustly minimality of the action.
%Under the assumption of extra regularity of the generators
%of $\Gamma$, Theorem~\ref{thmAA} implies the action is robustly
%ergodic with respect to the Lebesgue measure and conclude the proof
%of the proposition.
%
\end{proof}
\subsection{Ergodicity from globalized blending regions}
\label{ss:ergodic-blender} As a direct consequence of
Theorem~\ref{thmB} and Theorem~\ref{pro1} one gets the following:
% on the ergodicity of minimal actions having blending region.
\begin{corollary}
\label{corB} Every semigroup action of $C^{1+\alpha}$ conformal
diffeomorphisms having a globalized blending region is ergodic
with respect to  the Lebesgue measure.
\end{corollary}
Two global assumptions, conformality and $C^{1+\alpha}$-regularity, are involved in the result above.
In the sequel, using globalized blending
regions, we provide a new class of Lebesgue-ergodic minimal
expanding actions in which the assumptions hold only in a local region.

Let $B$ be a blending region for a semigroup $\Gamma$ as in
Definition~\ref{def-blending}. Note that in view of the covering
property~\eqref{cover00} and equality~\eqref{hutchinson}, one can
easily get the inclusion $B\subset \Delta$.

\begin{definition} We say that $B$ is a BDV-\emph{regular blending region} if
the associated contracting IFS is BDV-regular. We say that $B$ is
a \emph{$C^{1+\alpha}$ blending region} or a \emph{conformal
blending regions} if the generators of the associated contracting
IFS are $C^{1+\alpha}$ contracting  or conformal maps
respectively.
\end{definition}

Notice that in dimension one, every $C^{1+\alpha}$ blending region
is BDV-regular. Just like the one dimensional case, considering
$C^{1+\alpha}$-diffeomorphisms with complex eigenvalues one can
get a BDV-regular blending region in dimension two. These
particular cases provide $C^{1+\alpha}$-robustly conformal
blending regions.
%By conformal we mean that the generators are
%conformal maps, that is, their derivatives are positive scalar of
%orthogonal linear maps.
%and hence, robustly ergodic action with respect to the Lebesgue measure under local H\"older perturbations.
In higher dimensions, $C^{1+\alpha}$-conformal blending regions
are also BDV-regular. However, this regularity, i.e.
$C^{1+\alpha}$-conformality, does not persist in general under
perturbations of the associated contracting IFS.

%The following result provides examples of expanding minimal
%actions by not necessarily conformal $C^1$-diffeomorphisms
%where the minimality implies ergodicity only assuming that there
%exists a BDV-regular blending region.

\begin{theorem}
\label{pro2}
Consider a semigroup $\Gamma$ %generated by a family
of $C^1$-diffeomorphisms of a compact manifold. Suppose  that
there exists a globalized BDV-regular blending region $B$ for
$\Gamma$. Then, the action of $\Gamma$ is ergodic with respect to
the Lebesgue measure (and $C^1$-robustly minimal). Moreover, the
ergodicity persists under $C^1$-perturbations of the generators so
that $B$ remains as a BDV-regular blending region.
\end{theorem}

The main novelty of Theorem~\ref{pro2} is the local regularity
assumption, despite what was done in Theorem~\ref{thmB} (or
Corollary~\ref{corB}). That is, the regularity which is needed in
Theorem~\ref{pro2} limited to associated generators of the
contracting maps in the local blending region $B$. For instance,
to clarify the issue, consider an action $\Gamma$ on the circle
with a blending region $B$ whose associated contracting maps are
$C^{1+\alpha}$ in a neighborhood of $B$. Theorem~\ref{pro2}
implies the ergodicity of such action. Moreover, as we have
notified above, one can easily construct $C^{1+\alpha}$-robust conformal
blending regions in dimension two using complex eigenvalues.
Hence, as consequence of this theorem we get
Corollary~\ref{rem:ergodic-surface} and also the following:

 \begin{corollary}
\label{corC} Every forward and backward minimal semigroup action
(in particular minimal group actions) of $C^1$-diffeomorphisms of
a compact manifold with a $BDV$-regular blending region is
robustly\footnote{This robustness should be understood in the
sense describes in Theorem~\ref{pro2}} ergodic with respect to the
Lebesgue measure.
\end{corollary}

%\begin{corollary}
%\label{rem:ergodic-surface}
% Any compact surface admits a $C^{1+\alpha}$-robust minimal
%ergodic with respect to Lebesgue semigroup action whose generators
%are not necessarily conformal maps.
%\end{corollary}

To prove Theorem~\ref{pro2} we need the following lemma.
\begin{lemma}
\label{lem-densidad} Consider a semigroup $\Gamma$ of
$C^1$-diffeomorphisms of $M$. Suppose that there are an open set
$B\subset M$ and maps $S_1,\dots,S_n\in \Gamma$ such that
$M=S_1(B)\cup\cdots \cup S_n(B)$. Then, $B \cap DP(M\setminus
A)\not = \emptyset$ for all $\Gamma$-invariant set $A\subset M$
whose complementary has positive measure.
\end{lemma}
\begin{proof}
%First assume that $A$ is totally invariant, i.e., $g(A)=A$ for all
%$g\in \Gamma$, with positive Lebesgue measure and suppose that
%$DP(A)\cap B = \emptyset$. Hence, Lebesgue density theorem implies
%that $\lambda(A\cap B)=0$.  On the other hand, in view of
%invariance condition we have
%$$
%A= A\cap M= A\cap\big( \bigcup_{i=1}^n S_i(B)\big) \subset
%\bigcup_{i=1}^n S_i(A\cap B).
%$$
%By the quasi-invariance of the Lebesgue measure for smooth maps, it
%follows that $\lambda(S_i(A\cap B))=0$ for all $i=1,\ldots,n$ and
%therefore $\lambda(A)=0$ which is a contradiction.
%
Let $A$ be a forward invariant set, i.e., $g(A) \subset \Gamma$,
for all $g\in \Gamma$, such that its complementary $A^c=M\setminus A$ has
positive Lebesgue measure. Suppose that $DP(A^c) \cap B
=\emptyset$. The Lebesgue density theorem implies that
$\lambda(A^c\cap B)=0$. On the other hand, in view of the
(forward) invariance, we have
$$
A^c=A^c\cap M= A^c\cap \big(\bigcup_{i=1}^n S_i(B)\big) \subset
\bigcup_{i=1}^n S_i(A^c\cap B).
$$
Now,  the quasi-invariance of the Lebesgue measure implies
$\lambda(S_i(A^c\cap B))=0$, for all $i=1,\ldots,n$ and hence
$\lambda(A^c)=0$. This contradiction concludes the proof in this
case.  \end{proof}

Now, we are ready to prove Theorem~\ref{pro2}.
\begin{proof}[Proof of Theorem~\ref{pro2}]
%We are considering a group/semigruop $\Gamma$ generated by $C^1$-diffeomorphisms of a compact manifold with a BDV-regular blending region $B$.
In view of Theorem~\ref{pro1}, the action of $\Gamma$ on $M$ is
$C^1$-robustly minimal. It remains to prove the ergodicity of the
action with respect to the Lebesgue measure. Let $\mathcal{H}\subset \Gamma$
be a finite family of  generators of the associated BDV-regular
contracting IFS with the blending region~$B$. As mentioned before,
$$
  B\subset \Delta =
  \overline{\mathrm{Orb}^+_\mathcal{H}(x)},~~\text{for all $x\in \Delta$},
$$
%{\color{blue} I agree with you that the claim $B\subset \Delta$
%is not immediate from Lemma 3.3 (but for me is clear).
%I will explain you here as we can get this and you complete
%this claim with the details that you consider necessaries. \\
%The work of Hutchinson implies that $\Delta=
%\overline{\mathrm{Orb}^+_\mathcal{H}(x)}$. This follows since
%$\mathcal{H}^n(A) \to \Delta$ in the Hausdorff metric for all
%compact set $A\subset \overline{D}$ where
%$\mathcal{H}^n=\mathcal{H}\circ \mathcal{H}^{n-1}$ and
%$$
%\mathcal{H}(A)\eqdef h_1(A)\cup\dots\cup h_k(A).
%$$
%Indeed, it suffices to note that if $x\in \Delta$ then
%$\mathcal{H}^n(\{x\})\subset \Delta$ for all $n\in\mathbb{N}$.
%Now, we want to see that $B\subset \Delta$. The easiest way to
%see this is to note that $B\subset \mathcal{H}^n(\overline{B})$ for
%all $n\in\mathbb{N}$. Hence taking limit one concludes the required
%inclusion. From Lemma~3.3 one can also get this result but it is a
%little bit more complicate to write but the idea is the same.
%Take $x\in B$. For every $\varepsilon>0$, Lemma~3.3 implies that there exists $n_0\in \mathbb{N}$ such that $x\in ( \mathcal{H}^n(\overline{B})+\varepsilon)$ for all $n\geq n_0$. Recall that $A+\varepsilon$ denotes the set of point $y$ such that $d(a,y)<\varepsilon$ for some $a\in A$. Taking now limit as $n\to \infty$ and $\varepsilon\to 0$ we get that $x\in \Delta$. \\
%}
where $\mathrm{Orb}^+_\mathcal{H}(x)=\{h(x): h\in
\IFS(\mathcal{H})\}$ and $\Delta$ is the Hutchinson attractor. By
assumption, there exist maps $T_1,\dots,T_m,S_1,\dots,S_n\in
\Gamma$ such that
\begin{equation*}
M=T^{-1}_1(B)\cup \cdots \cup T^{-1}_m(B)=S_1(B)\cup \dots \cup
S_n(B).
\end{equation*} Consider a measurable set $A \subset M$
such that $g(A)\subset A$, for all $g\in \Gamma$.  Suppose that
$\lambda(A^c)>0$ where $A^c=M\setminus A$. We will prove that
$\lambda(A^c)=1$.

By Lemma~\ref{lem-densidad}, $B\cap DP(A^c)\not= \emptyset$. Since
$B$ is an open set in $\Delta$ one has that $DP(A^c)\cap
DP(\Delta)\not=\emptyset$. Hence, Lemma~\ref{lem:key} and the
invariance of $A^c$ imply that
$$
B\subset \Delta\,\mathring\subset\,
\mathrm{Orb}^-_\mathcal{H}(A^c)=A^c.
$$
%Since the action is minimal, Lemma~\ref{preliminar}, implies that there exists a finite set of maps $g_1,\dots, g_\ell \in \Gamma$
%such that
%$M= g_1^{-1}(B)
%\cup \dots\cup g_\ell^{-1}(B).
%$
Since $g^{-1}(A^c)\subset A^c$, for $g\in\Gamma$, then $
T_i^{-1}(B)\setminus A^c \subset T_i^{-1}(B\setminus A^c)$ for all
$i$. Now, the quasi-invariance of $\lambda$ for
$C^1$-diffeomorphisms implies that $\lambda(T_i^{-1}(B)\setminus
A^c)=0$ and so $\lambda(M \setminus A^c) =0$. This proves that
$\lambda(A^c)=1$, concluding the proof.
\end{proof}

\section*{Acknowledgements}
We are grateful to Jiagang Yang, Carlos Me\~nino, Meysam Nassiri
and Artem Raibekas for their useful comments and discussions.
During the preparation of this article the authors were partially
supported by the following fellowships: P.~G.~Barrientos by
MTM2011-22956 and MTM2014-56953-P project (Spain) and CNPq
post-doctoral fellowship 158839/2012-9 (Brazil); A. Fakhari by
grant from IPM (No. 94370127). The second author also thanks ICTP
for supporting through the association schedule.

\end{document}